\documentclass[journal]{IEEEtran}
\usepackage[intoc]{nomencl}

\usepackage{amsmath}
\usepackage{amsfonts}
\usepackage{graphicx}
\usepackage{epstopdf}
\usepackage{cite}
\usepackage{atbegshi}
\usepackage{float}
\usepackage{color,array}
\usepackage{amsthm}
\usepackage{algorithm}
\usepackage{algpseudocode}
\usepackage{pifont}
\AtBeginDocument{\AtBeginShipoutNext{\AtBeginShipoutDiscard}}

\theoremstyle{plain}

\theoremstyle{remark}
\newtheorem{rem}{\protect\remarkname}
\theoremstyle{definition}

\theoremstyle{definition}
\newtheorem{definition}{Definition}
\theoremstyle{plain}
\newtheorem{lem}{\protect\lemmaname}
\theoremstyle{plain}
\newtheorem{cor}{\protect\corollaryname}
\theoremstyle{plain}
\newtheorem{prop}{Proposition}

\providecommand{\corollaryname}{Corollary}
\providecommand{\lemmaname}{Lemma}
\providecommand{\problemname}{Problem}
\providecommand{\remarkname}{Remark}
\providecommand{\theoremname}{Theorem}

\begin{document}
\title{Robust Stability Analysis of DC Microgrids \\with Constant Power Loads}
\author{Jianzhe~Liu,~\IEEEmembership{Student Member,~IEEE,}
        ~Wei~Zhang,~\IEEEmembership{Member,~IEEE,}
        ~and Giorgio~Rizzoni,~\IEEEmembership{Fellow,~IEEE}}
\thanks{J. Liu is with both the Center for Automotive Research and the Department
of Electrical and Computer Engineering, the Ohio State University, Columbus, OH 43210. Email: liu.2430@osu.edu.}
\thanks{W. Zhang is with the Department
of Electrical and Computer Engineering, the Ohio State University, Columbus, OH 43210. Email:   zhang.491@osu.edu.}
\thanks{G. Rizzoni is with both the Center for Automotive Research and the Department of Mechanical and Aerospace Engineering, the Ohio State University, Columbus,
OH, 43210. Email: rizzoni.1@osu.edu.}
\allowdisplaybreaks[4]

\maketitle
\begin{abstract}
This paper studies stability analysis of DC microgrids with uncertain constant power loads (CPLs). It is well known that CPLs have negative impedance effects, which may cause instability in a DC microgrid. Existing works often study the stability around a given equilibrium based on some nominal values of CPLs. However, in real applications, the equilibrium of a DC microgrid depends on the loading condition that often changes over time. Different from many previous results, this paper develops a framework that can analyze the DC microgrid stability for a given range of CPLs. The problem is formulated as a robust stability problem of a polytopic uncertain linear system. By exploiting the structure of the problem, we derive a set of sufficient conditions that can guarantee robust stability. The conditions can be efficiently checked by solving a convex optimization problem whose complexity does not grow with the number of buses in the microgrid. The effectiveness and non-conservativeness of the proposed framework are demonstrated using simulation examples.
\end{abstract}



%
\section{Introduction}
\setcounter{page}{1}
A DC microgrid is a direct current power network that consists of locally-controlled sources and loads~\cite{4263070,1626398,7438879}. Directly connecting DC components allows for a simple integration of renewable generations and reduces unnecessary power conversion losses~\cite{Justo2013}. DC microgrids are finding various applications in more electric aircraft, naval ships, data centers, among others~\cite{Zhao201518,6813657,6816073,6287047,5200469,7101872}. With advanced power electronic devices, the transient behavior in the power outputs of many loads can be neglected, and these loads can be modeled as constant power loads (CPLs)~\cite{7182770,1658410,7469404,7401083}. 

Stability analysis problems for DC microgrids with CPLs have been studied in the literature. The CPLs are nonlinear and exhibit a negative impedance V-I characteristic, which may cause instability of a DC microgrid~\cite{6415284,6189081}. Some researchers study the stability of a DC microgrid around a fixed equilibrium calculated based on some nominal values of the CPLs~\cite{3248412,6909049,7328761,7182770,6031929,7798761,zonetti2016tool,7798761}. One major challenge lies in the nonlinearity contributed by the CPLs. To tackle the nonlinearity, some linearize a DC microgrid at a given equilibrium and analyze the resulting linearized system. In~\cite{3248412}, Nyquist stability criteria is used to determine the stability of a given equilibrium. In~\cite{zonetti2016tool}, the authors give conditions on the existence of an equilibrium for a DC microgrid. Then, the eigenvalues of the linearized system matrix is analyzed~\cite{6909049}. Others approximate the nonlinear DC microgrid system with linear systems~\cite{7182770,6031929}. In~\cite{7182770}, the nonlinearity is treated as a bounded noise so that a DC microgrid can be modeled by a linear system with additive uncertainty. In~\cite{6031929}, fuzzy modeling method is adopted to approximate the nonlinear system with a series of linear systems. With the resulting linear systems, estimations of the region of attraction around some known equilibria can be obtained by iteratively exploiting quadratic Lyapunov functions. Still others have devised different forms of Lyapunov-like functions to deal with the nonlinearity~\cite{7798761,6031929}. For example, in~\cite{7798761} a potential function based on Brayton-Moser method~\cite{669065} is used to prove the local stability of a given equilibrium. In the meantime, stability analysis problems may also arise in some DC microgrid consensus problems such as fair load sharing and voltage regulations~\cite{Zhao201518,de2016power,6816073,6287047}. The key to solving such problems is shaping the system equilibria to have the desired consensus. This shaping is usually realized by dynamic consensus controllers~\cite{7061540}. For example, in~\cite{6287047}, an distributed integral controller is used for voltage regulation. Still, the consensus cannot be reached if the equilibrium is not stable. Both linearization~\cite{6816073} and Lyapunov-like method~\cite{de2016power} has been used for such stability analysis. In~\cite{6816073}, the linearization method is used with a given equilibrium for average voltage regulation and current sharing. In~\cite{de2016power}, Bregman storage functions~\cite{7798653} are utilized to solve a power sharing problem. The steady state of a DC microgrid with CPLs is studied in the paper. It is represented by a static model. The static model can be physically interpreted as a resistive circuit, which has no dynamic components such as inductors or capacitors. A dynamic consensus controller is used to make the resistive circuit reach some desired steady states. This act is equivalent to using a dynamic distributed algorithm to find the solution of algebraic equations. Some conditions with respect to resistance and negative impedance are shown to determine the convergence of the algorithm as well. Nevertheless, since the paper only considers a static model the convergence of the algorithm cannot ensure the stability of a DC microgrid with dynamic components. Modifications of the storage function similar to that of~\cite{7798761} could be used, but the knowledge of the equilibrium is still needed for the stability analysis.

Despite the rich literature on DC microgrids, existing works have several limitations. First, existing stability analysis methods are based on a fixed equilibrium and require to know its exact value a priori~\cite{3248412,6909049,7328761,7182770,6031929,7798761,zonetti2016tool,7798761,7061540}. However, in real applications, the system equilibrium depends on uncertain CPLs and changes over time~\cite{6759911,6887330}. Second, existing works mostly consider single-bus DC microgrids~\cite{6909049,6031929,7182770,7469404}. It is unclear how the proposed methods can be extended to more general topology that often arises in real applications~\cite{6575114}. Third, many methods proposed in the literature are ad hoc in nature and are specific to a particular application scenario~\cite{6816073,6909049,6287047,6678794,de2016power}. There lacks a systematic method to analyze the stability of the closed loop system, especially when uncertain CPLs and general topology are considered.

In this paper, we formulate and study the robust stability problem of a general DC microgrid with uncertain CPLs. We assume that the power of each CPL may take arbitrary values within a given interval. As a result, the vector of the overall CPL power lies in a polytopic uncertainty set. Different CPL vectors in the uncertainty set may lead to different system equilibria. We call the system locally robustly stable if all the resulting equilibria are locally exponentially stable. We show that checking robust stability of a general DC microgrid can be formulated as a robust stability analysis problem of linear systems with polytopic uncertainties on system matrices. Existing results on linear uncertain systems~\cite{Scherer2001361,486646,1299019,914559,940939} can be used to derive sufficient locally robust stability conditions; however, these conditions are computationally challenging to verify and may be over conservative. For example, in~\cite{7223059} the authors use state space sampling method to have a polytopic set with $n$ vertices and solve an LMI problem with $2^n$ constraints. In this paper, we show that for our DC microgrid robust stability problems, the parameter uncertainty lies in the diagonal entries of the system matrix. We take advantage of this structure and derive a set of sufficient conditions that can be efficiently checked to guarantee locally robust stability. The conservativeness issue is also discussed in the paper, we obtain an approximated quantification of the conservativeness of the proposed work. By virtue of the quantified conservativeness, a quantitative relationship between the CPL power and voltage is derived, which is useful for DC microgrid operations. For example, with given CPL power range one can compute a constraint for CPL voltage for stability. When the voltage constraint is satisfied during operation, the DC microgrid can be guaranteed to be locally  exponentially stable. To our knowledge, the robust stability problem formulated in this paper has not been studied in the DC microgrid literature, and our results provide new insights for stability analysis of DC microgrids with uncertain CPLs.

The rest of the paper is organized as follows, Section~\ref{sec:model} presents a model for a general DC microgrid with CPLs, Section~\ref{sec:rob} develops the locally robust stability framework, Section~\ref{sec:case} uses a simulation study to verify the main results, and Section~\ref{sec:conclusion} concludes the paper and discusses future research directions.

\section{General DC Microgrid Modelling}\label{sec:model}
In this paper, we consider an $n$-bus DC microgrid shown in Fig.~\ref{fig:topo}. These buses form an undirected graph defined by $\mathcal{G}=(\mathcal{V},\mathcal{E})$ where each element in the vertex set $\mathcal{V}$ represents one DC microgrid bus, and each element in the edge set $\mathcal{E}$ represents one transmission line between two buses. If $(k,j)\in \mathcal{E}$, we say that the $j^{\text{th}}$ bus is a neighboring bus of the $k^{\text{th}}$ bus. Let the set $\mathcal{N}_k$ have the indices of all the neighboring buses of the $k^{\text{th}}$ bus. Let $W\in \mathbb{R}^{n\times n}$ be the adjacency matrix. 
\vspace{-0.3cm}
\begin{figure}[H]
\centering
\includegraphics[width=0.47\textwidth]{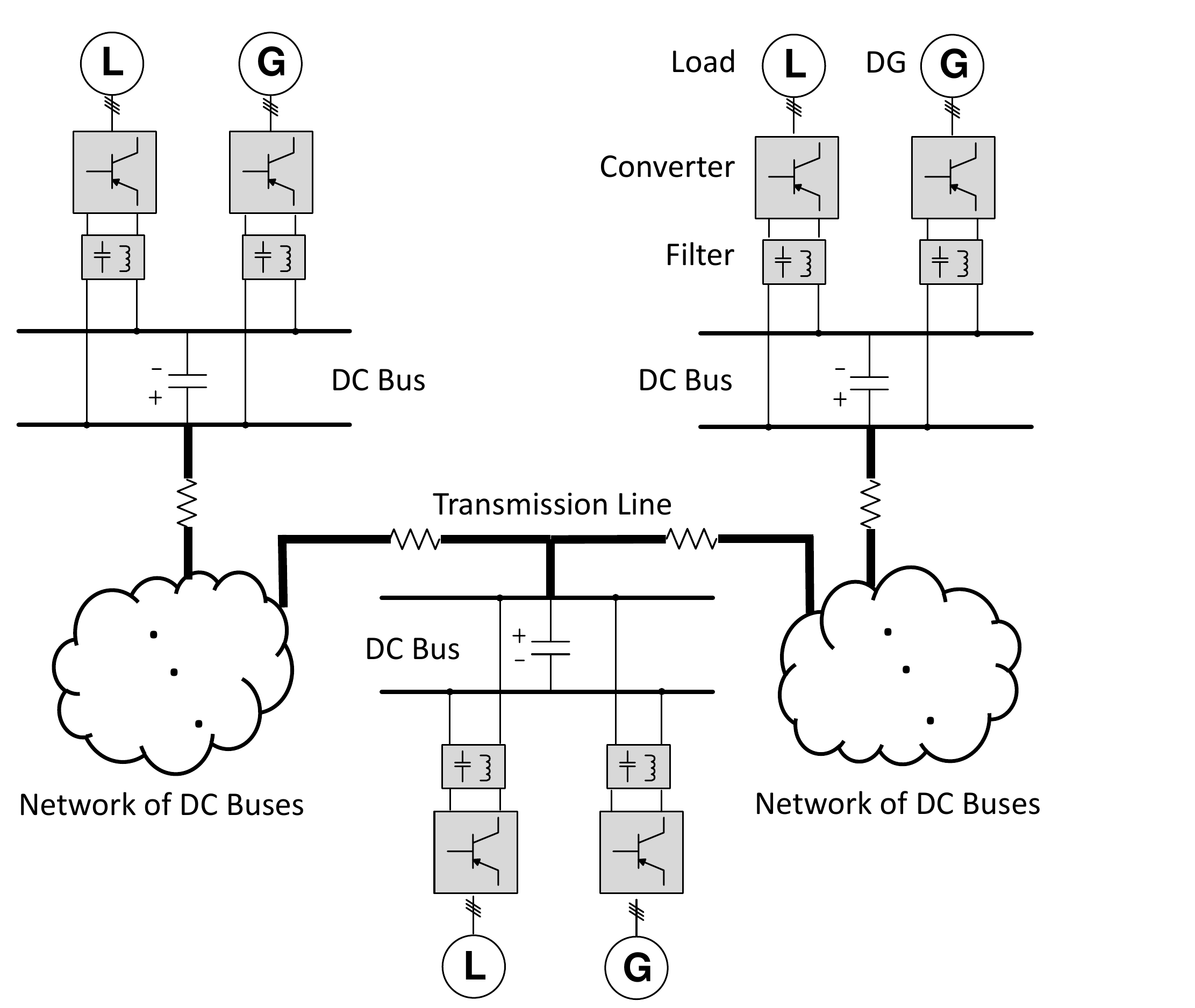}
\caption{\label{fig:topo}The $n$-bus DC microgrid}
\end{figure} 

\vspace{-0.5cm}
\begin{figure}[H]
\centering
\includegraphics[width=0.5\textwidth]{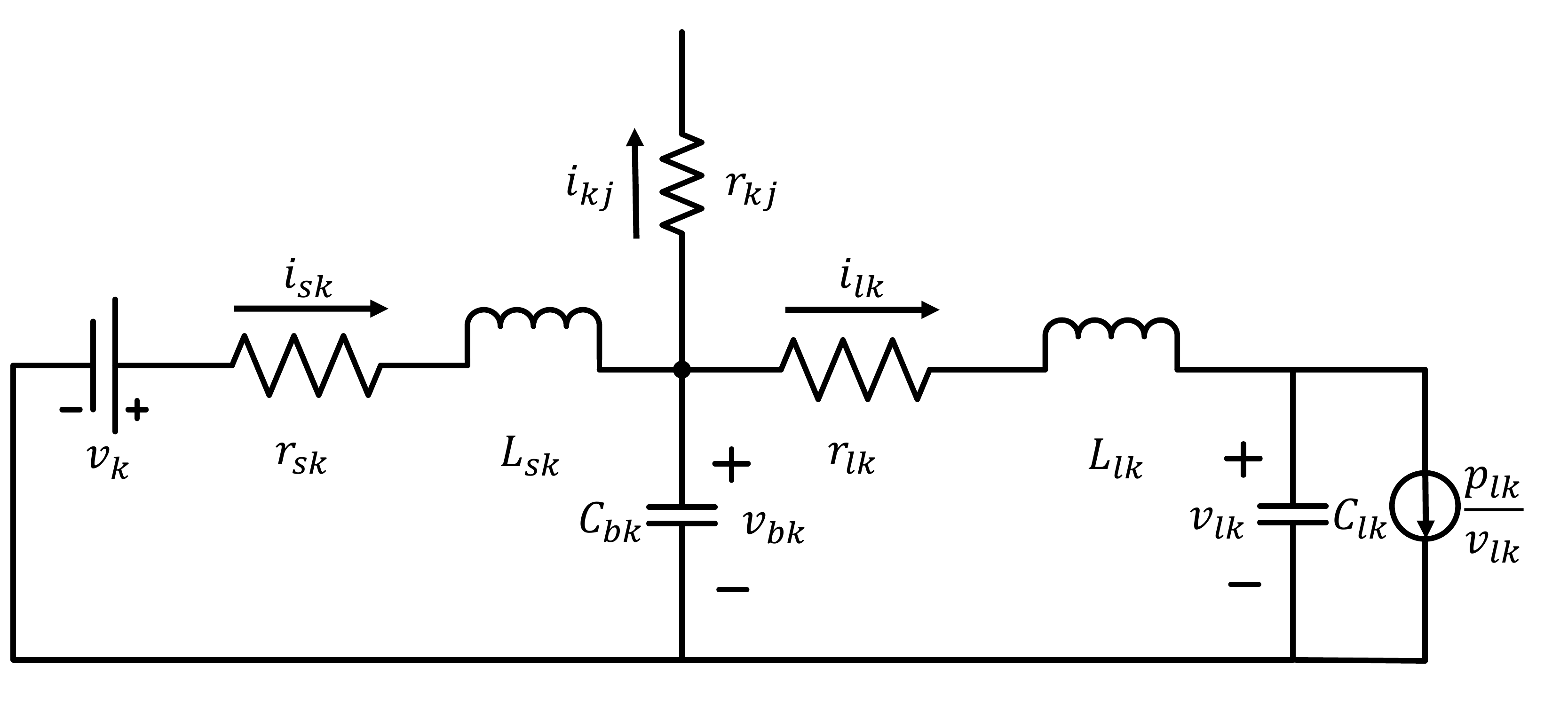}
\caption{\label{fig:1bus}The $k^{\text{th}}$ bus of the DC microgrid}
\end{figure}

To simplify discussion, we assume that there is one controllable voltage source and one constant power load (CPL) on each bus. The equivalent circuit model of the $k^{\text{th}}$ bus is depicted in Fig.~\ref{fig:1bus}. In this section, we will first develop a dynamic model for bus $k$, and then construct the overall dynamic model for the entire microgrid. Note that the modeling method can be easily extended for DC microgrids with multiple (or none) sources or loads on some buses.

\subsection{Dynamic Model for Bus $k$}
\subsubsection{Voltage Source} We assume that each bus has one controllable voltage source with standard V-I droop control~\cite{5546958}. Let $v_k$ and $i_{sk}$ be the terminal voltage and output current of the source on Bus $k$, respectively. The source connects with the rest of the microgrid through a resistor and an inductor. The current $i_{sk}$ is the state variable for the source. The dynamics of $i_{sk}$ are given by,
\begin{equation}
\left\{ \begin{array}{l}
L_{sk}\frac{di_{sk}}{dt}=v_k-r_{sk} i_{sk}-v_{bk}\\
v_k=v_k^{\text{ref}}-d_k i_{sk}
\end{array},\right.\label{eq:isi}
\end{equation}
where $L_{sk}$ and $r_{sk}$ are the source inductance and resistance, $v_{bk}$ is the voltage of the $k^{\text{th}}$ DC bus, $v_k^{\text{ref}}$ and $d_k$ are the droop reference and droop gain, respectively. Here, we have ignored the internal dynamics of the power electronic devices, which is a standard assumption for microgrid stability and control problems~\cite{Zhao201518}.  


\subsubsection{CPL} The load is a CPL and is modeled as a current sink whose current injection is given by the CPL power divided by the CPL voltage~\cite{1658410,6415284,6189081}. The CPL is connected with the microgrid through an equivalent RLC filter. Let $v_{lk}$ and $i_{lk}$ be the terminal CPL voltage and the current injection into the filter, respectively. The CPL voltage $v_{lk}$ equals the voltage across the capacitor in parallel with the current source. The current $i_{lk}$ flows from the DC bus capacitor and feeds the load through an inductor and a resistor. Here, $v_{lk}$ and $i_{lk}$ are two state variables of the CPL. The dynamics of the state variables are given by,
\begin{equation}
\left\{ \begin{array}{l}
L_{lk}\frac{di_{lk}}{dt}=v_{bk}-r_{lk}i_{lk}-v_{lk}\\
C_{lk}\frac{dv_{lk}}{dt}=i_{lk}-\frac{p_k}{v_{lk}}
\end{array},\right.
\label{eq:cpli}
\end{equation}
where $L_{lk}$, $r_{lk}$, $C_{lk}$ are load inductance, resistance, and capacitance, respectively, and $p_k$ is the power of the CPL. In~\eqref{eq:cpli}, the current injection into the current sink is $p_k/v_{lk}$, making the power injection to be $p_k$. We assume that $p_k$ is uncertain and may take any value in the interval $[\underline{p_k},\overline{p_k}]$ where $\overline{p_k}\geq \underline{p_k}\geq 0$. One can see that the CPL introduces an additive nonlinear term $p_k/v_{lk}$ in~\eqref{eq:cpli}.

\subsubsection{DC Link Capacitor} The DC bus voltage of the $k^{\text{th}}$ bus, $v_{bk}$, is the voltage across the DC link capacitor. It is determined by the source current, the load current, and the transiting currents between the neighboring buses. The DC bus voltage, $v_{bk}$, is the state variable for the DC link capacitor, and its dynamics are given by,
\begin{equation}
\left\{ \begin{array}{l}
C_{bk}\frac{dv_{bk}}{dt}=i_{sk}-i_{lk}-\sum_{j\in \mathcal{N}_k}i_{kj}\\
i_{kj}=\frac{v_{bk}-v_{bj}}{r_{kj}}
\end{array},\right.\label{eq:busi}
\end{equation}
where $\mathcal{N}_k$ is the set of indices of all the neighboring buses of the $k^{\text{th}}$ bus, $C_{bk}$ and $r_{kj}$ are the DC link capacitance of the $k^{\text{th}}$ bus and the equivalent resistance of the transmission line between the $k^{\text{th}}$ and $j^{\text{th}}$ buses, respectively. Note that $\mathcal{N}_k$ is determined by the microgrid network topology.
\subsection{$n$-bus DC Microgrid}
We now combine the individual models of the sources, CPLs, and the DC link capacitors on all the $n$ buses to obtain the overall model of the DC microgrid.

Let $i_x$, $v_x\in \mathbb{R}^{2n}$, $p$, $v$, and $v^{\text{ref}}$ $\in \mathbb{R}^n$, be the vectors of all the source and CPL currents, DC bus and CPL voltage, load power, source voltage, and droop references, respectively. They are given as follows,
\begin{align}
&i_x\!\!=\!\![i_{s1},\!\cdots\!,i_{sn},i_{l1},\!\cdots\!,i_{ln}]^T\!,\! v_x\!\!=\!\![v_{\text{B}1},\cdots,v_{\text{B}n},v_{l1},\cdots,v_{ln}]^T\!\!, \nonumber\\
&p\!\!=\!\![p_1\!,\!\cdots\!,\!p_n]^T, v\!\!=\!\![v_{1}\!,\!\cdots\!,\!v_{n}]^T\!, v^{\text{ref}}\!\!=\!\![v^{\text{ref}}_1\!,\!\cdots\!,\!v^{\text{ref}}_n]^T.\nonumber
\end{align}
In addition, define a vector $x\in \mathbb{R}^{4n}$ and a nonlinear function $h(\cdot,\cdot): \mathbb{R}^n\times \mathbb{R}^{4n}\to \mathbb{R}^n$ as follows,
\begin{align}
x=[i_x^T,v_x^T]^T,\quad h(p,x)=[-\frac{p_1}{v_{l1}},\cdots,-\frac{p_n}{v_{ln}}]^T.\nonumber
\end{align}
The vector $x$ stacks all the current and voltage states of the components, and the nonlinear function $h(p,x)$ contains the additive nonlinear terms in~\eqref{eq:cpli}.

Let $D_k\in \mathbb{R}^{4n\times 4n}$, and only its $k^{\text{th}}$ diagonal entry is non-zero and is 1. 

With the above notation, the overall microgrid dynamics can be written in the following compact form:
\begin{equation}
\dot{x}=Ax+Bv^{\text{ref}}+Ch(p,x),\quad A=A_0-B\sum_{k=1}^n d_kD_k, \label{eq:gencls}
\end{equation}
where the matrices $A$, $A_0\in \mathbb{R}^{4n\times 4n}$, and $B$, $C\in \mathbb{R}^{4n\times n}$. The matrix $A_0$ is the system matrix of the open loop DC microgrid without the V-I droop control, and it is determined by the system parameters and topology. When the droop control is utilized, the matrix $A$ is the system matrix of the closed loop system. Other than the parameters and topology, $A$ depends on the droop gains as well.

We now briefly discuss how the matrices $A$, $B$, and $C$ can be constructed from individual bus models~\eqref{eq:isi}-\eqref{eq:busi} as well as the system topology. Substituting $i_{sk}=x_k$, $v_{bk}=x_{2n+k}$ into \eqref{eq:isi} yields,
\begin{equation}
\frac{di_{sk}}{dt}=\dot{x}_k=\frac{1}{L_{sk}}v_k^{\text{ref}}-\frac{r_{sk}+d_k}{L_{sk}}x_k-\frac{1}{L_{sk}}x_{2n+k},\nonumber
\end{equation}
thus, $A_{kk}=-(r_{sk}+d_k)/L_{sk}$, $A_{k(2n+k)}=-1/L_{sk}$, $B_{kk}=1$. In addition, $C=\sum_{k=1}^n1/C_{lk}D_{3n+k}$ from~\eqref{eq:cpli}. Furthermore, from $\eqref{eq:busi}$, if $j\in\mathcal{N}_k$, then $A_{(2n+k)(2n+j)}=1/r_{kj}$, and $A_{(2n+k)(2n+k)}=-|\mathcal{N}_k|/r_{kj}$, where $|\cdot|$ gives the cardinality of a set. This demonstrates how the matrix $A$ depends on the system topology.

\begin{rem}
In deriving the microgrid model~\eqref{eq:gencls}, we have adopted a commonly used resistive transmission line model~\cite{Zhao201518}. The modeling method can be applied to other transmission line types. For example, regarding transmission line $\Pi$-model~\cite{kundur1994power} that has an equivalent capacitor on each end of the line, the capacitors are in parallel with the DC bus link capacitors and can be directly combined with them. Hence, the resulting model is still in the form of~\eqref{eq:gencls} with modified $C_{bk}$ parameters. For $\Pi$-model with equivalent line inductance~\cite{kundur1994power}, another state that represents the line current can be added. It will also result in a microgrid model similar to~\eqref{eq:gencls} with an increased state dimension. 
\end{rem} 

\begin{rem}
Recently, different control laws have been proposed to replace the V-I droop control~\cite{6816073,6813657,7182770}. The modeling method presented in this paper and the corresponding stability results can be easily extended to DC microgrid with these control laws. In Section~\ref{sec:robdiscus}, we exemplify it by applying the proposed methods to a DC microgrid with distributed control.
\end{rem}

\begin{rem}
Depending on the control of the power electronic devices, there are different types of sources~\cite{720325,903988}. With controlled current output, one source can be seen as a current source~\cite{4812332}. If the output current is regulated to be linear in the system state, the modeling method of the source can be directly applied. For instance, with current droop control~\cite{720325}, the current output of the source is linear in the output voltage, $i_{sk}=i_{sk}^{\text{ref}}-d_kv_{k}$, where $i_{sk}^{\text{ref}}$ and $d_k$ are constants, and source voltage $v_{k}$ is measured at each instant. Thus, the doop controlled current source can be modeled as a constant current source, $i_{sk}^*$, in parallel to a virtual resistance, $1/d_k$. We can use the following model to represent the source: 
\begin{equation}
\left\{ \begin{array}{l}
L_{sk}\frac{di_{sk}}{dt}=v_k-r_{sk} i_{sk}-v_{bk}\\
i_{sk}=i_{sk}^{\text{ref}}-d_k v_{k}
\end{array},\right.\nonumber
\end{equation}
and this linear model can be directly incorporated into model~\eqref{eq:gencls}. Meanwhile, with output power under control a source can be seen as a power source~\cite{903988}. Let $p_{sk}$ be the power output of the source. If a power source is involved in the system, it can be modeled as a current source with nonlinear power output in the state, $p_{sk}/v_{k}$, as we model the CPL. Hence, the main idea and results of this paper can be extended to the power source as well. 
\end{rem}

\section{Robust Stability Framework}\label{sec:rob}
\subsection{Robust Stability under Uncertain CPLs}\label{sec:robprob}
The stability analysis of system~\eqref{eq:gencls} is crucial for DC microgrid applications since the CPLs might cause instability. In this paper, the uncertain CPL power $p$ is assumed to be physically bounded in a polytopic set $\mathcal{P}$ defined as follows,
\begin{align}
\mathcal{P}=\left\{p:p_k\in [\underline{p_k},\overline{p_k}],k=1,\cdots,n\right\}.\nonumber
\end{align}

We aim at analyzing the stability of system~\eqref{eq:gencls} with uncertain $p\in \mathcal{P}$. The challenge of the analysis lies in that the equilibria of the system are difficult to find~\cite{4433425}. With $n$ CPLs, one needs to solve $n$ quadratic equations to obtain the equilibria of system~\eqref{eq:gencls}. In general, there are $2^n$ solutions and there are no general results to determine which one is the actual equilibrium the system will converge to~\cite{7108029}. Furthermore, the solution of the equations changes with respect to $p$. With uncertain CPL power profile $p$, the set of possible equilibria cannot be characterized for a general DC microgrid. 

Due to the lack of a predetermined equilibrium, we study the stability of all the possible equilibria that are feasible for the microgrid. For a vector $p\in \mathcal{P}$, let $x_e\in \mathbb{R}^{4n}$ be an equilibrium of system~\eqref{eq:gencls}. At $x_e$, let the steady state voltage of the $k^{\text{th}}$ CPL be $v_{lk}^e$, define a vector $v_{l}^e=[v_{l1}^e,\cdots,v_{ln}^e]^T$. Assume that an operationally feasible CPL voltage vector has to lie inside a known constraint set $\mathcal{V}_l^e$ in the following form,
\begin{equation}
\mathcal{V}_l^e=\left\{ v_l^e\in \mathbb{R}^{n}:v_{lk}^e\in [\underline{v_{lk}^e},\overline{v_{lk}^e}],k=1,\cdots,n\right\},\nonumber
\end{equation}
where $\overline{v}^e_{lk}\geq \underline{v^e_{lk}}\geq 0$.

As $p$ varies in $\mathcal{P}$, $x_e$ changes accordingly. Some equilibria might not be physically admissible by the circuit. We thus focus on the following set of equilibria whose corresponding CPL voltage lies in $\mathcal{V}_l^e$, 
\begin{equation}
\mathcal{X}_e\!(\!\mathcal{P},\!\!\mathcal{V}_l^e\!)\!\!=\!\!\left\{\!\!x_e\!\!\in\!\! \mathbb{R}^{4n}\!\!:\! \!Ax_e\!\!+\!\!Bv^{\text{ref}}\!\!\!+\!\!Ch(p,x_e)\!=\!0,\!p\!\in\!\mathcal{P},\!v_l^e\!\in\!\mathcal{V}_l^e\!\right\}\!\!.\nonumber
\end{equation}

\begin{definition}\label{def:robstab}
System~\eqref{eq:gencls} is said to be locally robustly stable if any equilibrium in the set $\mathcal{X}_e(\mathcal{P},\mathcal{V}_l^e)$ is locally exponentially stable.
\end{definition}

The above defined locally robust stability is not exactly the same as the definition used in the classical robust control literature~\cite{leitmann1979guaranteed,50357}. The classical robust control typically studies the stability of one predetermined equilibrium under uncertainty. Note that such an equilibrium is not available for DC microgrids with uncertain CPLs, we study all the equilibria lying in $\mathcal{X}_e(\mathcal{P},\mathcal{V}_l^e)$. 
\subsection{Robust Stability Analysis}\label{sec:robresult}
This subsection develops methods to analyze the locally robust stability defined in Definition~\ref{def:robstab}. We derive sufficient conditions to guarantee the locally robust stability of system~\eqref{eq:gencls}. Computationally efficient convex optimization problems are formulated to facilitate the analysis.

Let $x_e$ be an arbitrary equilibrium in $\mathcal{X}_e(\mathcal{P},\mathcal{V}_l^e)$. Linearizing system~\eqref{eq:gencls} around $x_e$ yields,
\begin{equation}
\dot{z}=A_zz, \quad A_z=A+\sum_{k=1}^n\delta_kD_{3n+k},\label{eq:genlin}
\end{equation}
where $z\in \mathbb{R}^{4n}$ is the state variable of the linearized system, $A_z\in \mathbb{R}^{4n\times 4n}$ is the linearized system matrix, and $\delta_k=p_k/(C_{lk}(v^e_{lk})^2)$, which is obtained by linearizing the $k^{\text{th}}$ entry of $Ch(p,x)$ at $x_e$. Since $p_k$ and $v_{lk}^e$ vary in $[\underline{p_k},\overline{p}_k]$ and $[\underline{v_{lk}^e},\overline{v}_{lk}^e]$, respectively, $\delta_k$ takes value in an interval, $[\underline{\delta_k},\overline{\delta}_k]$, where $\underline{\delta_k}=\underline{p_k}/(C_{lk}(\overline{v^e_{lk}})^2)$ and $\overline{\delta}_k=\overline{p_k}/(C_{lk}(\underline{v^e_{lk}})^2)$. Note that the overall uncertainty of $p_k$ and $v_{lk}^e$ is captured by $\delta_k$.

The system matrix $A_z$ is an uncertain matrix, and it belongs to a set $\mathcal{A}_z$ defined as follows,
\begin{equation}
\mathcal{A}_z\!\!=\!\!\left\{ \!\!A_z\!\!\in\!\!\mathbb{R}^{4n\times 4n}\!:\! A_z\!\!=\!\!A\!\!+\!\!\sum_{k=1}^n \delta_kD_{3n+k},\delta_k\in [\underline{\delta_k},\overline{\delta}_k]\!\!\right\}.\label{eq:Az}
\end{equation}
The set $\mathcal{A}_z$ is a polytopic set, and it has $2^n$ vertices. Let the matrix $A_{j}^v$ be the $j^{\text{th}}$ vertex of the set $\mathcal{A}_z$. It can be written in the form: $A^v_j=A+\sum_{k=1}^n\delta_kD_{3n+k}$ with $\delta_k=\overline{\delta}_k$ or $\delta_k=\underline{\delta_k}$. 

Recall that a matrix is Hurwitz stable if all of its eigenvalues have negative real parts. Given a CPL uncertainty set $\mathcal{P}$ and a voltage constraint set $\mathcal{V}_{l}^e$, all the possible equilibria lie inside $\mathcal{X}_e(\mathcal{P},\mathcal{V}_l^e)$ and all the possible linearized system matrices lie inside $\mathcal{A}_z$. Therefore, system~\eqref{eq:gencls} is locally robustly stable if all the matrices in $\mathcal{A}_z$ are Hurwitz stable. Notice that there are infinitely many elements in the set $\mathcal{A}_z$. As a standard result in linear uncertain systems, to ensure stability of all the matrices in $\mathcal{A}_z$, it suffices to have a common Lyapunov function for all the vertices of $A_z$~\cite{7223059,mansour1988sufficient,250472}. 

For a matrix $M$, let $M\succ 0$ and $M\prec 0$ represent that $M$ is positive definite and negative definite, respectively, let $M\succeq 0$ and $M\preceq 0$ represent that $M$ is positive semidefinite and negative semidefinite, respectively. 

\begin{lem}\label{lem:2nLMI}
System~\eqref{eq:gencls} is locally robustly stable if $\exists P=P^T\succ 0$ satisfying
\begin{align}
PA_j^v+(A_j^v)^TP\prec 0,\quad j=1,\cdots,2^n.\label{eq:2n}
\end{align}
\end{lem}


%
%
\begin{rem}
If constraints~\eqref{eq:2n} are all feasible for some $P=P^T\succ 0$, $U(z)=z^TPz$ is a common Lyapunov function for the vertices in $\mathcal{A}_z$. Since $\mathcal{A}_z$ is a polytopic set, any matrix $A_z\in \mathcal{A}_z$ can be expressed as a convex combination of its vertices~\cite{mansour1988sufficient}. Thus, the function $U(z)$ is a common Lyapunov function for all the matrices in $\mathcal{A}_z$ as well, and any matrix $A_z\in \mathcal{A}_z$ is Hurwitz stable by Lyapunov method for stability~\cite{250472}. 
\end{rem}

Checking the condition given in Lemma~\ref{lem:2nLMI} is a linear matrix inequalities (LMIs) feasibility problem, which can be checked efficiently~\cite{boyd2004convex}. There are $2^n$ LMIs involved in the feasibility problem. Even though LMIs can be checked efficiently, when $n$ is nontrivial the computation burden could still be huge. This limits the applicability of Lemma~\ref{lem:2nLMI} for general DC microgrid stability analysis. 

This limitation can be eliminated by making use of the special structure of the matrix $A_z$. One can observe that each $\delta_k$ lies on one diagonal entry of $A_z$. We take advantage of this structure and transform the condition from checking $2^n$ LMIs to solving a convex optimization problem whose complexity does not increase with $n$.

Let $\overline{A}_z\in \mathcal{A}_z$ be the element-wise maximum of all the matrices in $\mathcal{A}_z$, i.e. it is given by
\begin{equation}
\overline{A}_z=A+\sum_{k=1}^{n}\overline{\delta}_kD_{3n+k}.\nonumber
\end{equation}
Let the largest value of all $\overline{\delta}_k$, $k=1,\cdots,n$, be $\overline{\delta}_{\text{max}}$. It is worth pointing out that $\overline{A}_z$ and $\overline{\delta}_{\text{max}}$ are deterministic and can be obtained easily. In addition, let $\Delta \delta_k=\overline{\delta}_k-\delta_k$ and let $Q_j$ be a matrix defined as follows,
\begin{equation}
Q_j=P\sum_{k=1}^n \Delta \delta_k D_{3n+k}+(\sum_{k=1}^n \Delta \delta_k D_{3n+k})^TP.\label{eq:thmQ}
\end{equation}

\begin{lem}\label{lem:cvx}
System~\eqref{eq:gencls} is locally robustly stable if $\exists P=P^T\succ 0$, $\gamma>0$, and $t>0$ such that the following optimization problem is feasible, 
\begin{align}
&\min_{P,t,\gamma} t-\gamma  \label{eq:LMImin}\\
\textnormal{subj. to: }&P\overline{A}_z+\overline{A}_z^TP\preceq -\gamma I_{4n},\nonumber\\
&P\preceq tI_{4n},\quad 2t\cdot \overline{\delta}_{\text{max}}< \gamma,\nonumber\\
&P=P^T,P\succ 0,\gamma>0,t>0.\nonumber
\end{align}
\end{lem}

\begin{rem}
The proof of Lemma~\ref{lem:cvx} is reported in Appendix A. Instead of solving the feasibility problem given in Lemma~\ref{lem:2nLMI} that involves $2^n$ LMIs, the condition derived in Lemma~\ref{lem:cvx} allows us to verify the locally robust stability by solving a convex optimization problem with only one set of LMIs. One explanation is that for Lemma~\ref{lem:2nLMI}, we need to check the stability of $2^n$ DC microgrids with different CPL power and steady state voltage. In Lemma~\ref{lem:cvx}, we only check the stability of one DC microgrid in an extreme case where each CPL is operated at its own maximum power and the lowest steady state voltage. Meanwhile, we find an upper bound on the largest difference between the parametric uncertainty in this case and those in the others. The satisfaction of this upper bound along with the stability of the DC microgrid in the extreme case guarantees the locally robust stability of the DC microgrid. 
\end{rem}

The feasibility of problem~\eqref{eq:LMImin} guarantees the locally robust stability of system~\eqref{eq:gencls}. The cost function of the problem is motivated by the structure of the constraints. When there exists multiple solutions, cost function~\eqref{eq:LMImin} selects the ``optimal'' solution that minimizes the function $t-\gamma$. The cost function, however, does not affects the feasibility of the problem, and any other convex cost functions can be utilized. 

Problem~\eqref{eq:LMImin} only checks the Hurwitz stability of one matrix $\overline{A}_z$ and has a constraint on $\overline{\delta}_{\text{max}}$. We find that by lower bounding the smallest eigenvalue of every $Q_j$, the Hurwitz stability of $\overline{A}_z$ ensures the locally robust stability. This upper bound can be ensured by upper bounding the largest possible parameter uncertainty $\overline{\delta}_{\text{max}}$. The price we pay for such a complexity reduction is the increased conservativeness coming from the difference between the original lower bound and the uniform constraints. In other words, there can be locally robustly stable microgrids that satisfy the conditions in Lemma~\ref{lem:2nLMI}, but fail to satisfy the condition in Lemma~\ref{lem:cvx}. The conservativeness can be reduced when we directly focus on the smallest eigenvalue of each matrix $Q_j$ in the program by introducing constraints such as $-\gamma I_{4n}\prec Q_j$, $j=1,\cdots,2^n$. Nevertheless, this again increases the computational complexity by adding another $2^n$ constraints. We next show that this can be accomplished by adding only a polynomial number of constraints.

To simplify notation, we define a matrix $G_k$ as follows,
\begin{equation}
G_k=PD_{3n+k}+D_{3n+k}P.\nonumber
\end{equation}

\begin{prop}\label{prop:lescsv}
System~\eqref{eq:gencls} is locally robustly stable if $\exists P=P^T\succ 0$, $\gamma_0,\cdots,\gamma_{n}>0$, $\eta_1,\cdots,\eta_n>0$, such that the following optimization problem is feasible, 
\begin{align}
&\min_{P,\gamma_0,\cdots,\gamma_n,\eta_1,\cdots,\eta_n} -\sum_{k=0}^n\gamma_k \label{eq:LMImin2}\\
\textnormal{subj. to: }& P\overline{A}_z+\overline{A}_z^TP\preceq -\gamma_0 I_{4n}, \quad \sum_{k=1}^{n}\eta_k=\gamma_0, \nonumber \\
&  \overline{\delta}_{k}\cdot G_k\succeq (\gamma_k-\eta_k)I_{4n},  \quad k=1,\cdots,n,\nonumber\\
& \gamma_k>0, \eta_k>0, \quad k=1,\cdots,n,\nonumber\\
&  P=P^T,P\succ 0,\gamma_0>0.\nonumber
\end{align}
\end{prop}

\begin{proof}
We assume that problem~\eqref{eq:LMImin2} is feasible. Let $P=P^T\succ 0$, $\gamma_0,\cdots,\gamma_n>0$, and $\eta_1,\cdots,\eta_n>0$ be the solution of the problem. From Lemma~\ref{lem:cvx}, to prove that system~\eqref{eq:gencls} is locally robustly stable, we only need to show that the smallest eigenvalue of $Q_j$ is greater than $\gamma_0$. 

The matrix $Q_j$ can be expressed in terms of $G_k$ and $\Delta \delta_k$ as follows,
\begin{equation}
Q_j\!=\!\!P\!\sum_{k=1}^n \!\Delta \delta_k D_{3n\!+\!k}\!+(\!\sum_{k=1}^n\! \Delta \delta_k D_{3n+k})^T\!P\!\!=\!\!\sum_{k=1}^n\!\!\Delta \delta_k G_k.\label{eq:propQj}
\end{equation}
Since $Q_j$ is symmetric, its smallest eigenvalue is given by $\min_{y^Ty=1}y^TQ_jy$. Substituting~\eqref{eq:propQj} into $\min_{y^Ty=1}y^TQ_jy$ yields,
\begin{equation}
\min_{y^Ty=1}y^TQ_jy=\min_{y^Ty=1}y^T\left(\sum_{k=1}^n\!\Delta \delta_k G_k\right)y.\nonumber
\end{equation}

By Weyl's inequality,
\begin{align}
\min_{y^Ty=1}y^TQ_jy\geq \sum_{k=1}^n \Delta \delta_k \!\!\min_{y^Ty=1}y^T G_k y
\nonumber.
\end{align}
From the inequality $\overline{\delta}_{k}\cdot G_k\succeq (\gamma_k-\eta_k)I_{4n}$, the smallest eigenvalue of $G_k$ is lower bounded, $\min_{y^Ty=1}y^T G_k y\geq \Delta \delta_k(\gamma_k-\eta_k)/\overline{\delta}_k$. In addition, seeing that $\gamma_k>0$, and $0\leq \Delta \delta_k/\overline{\delta}_k\leq 1$,
\begin{equation}
\min_{y^Ty=1}y^TQ_jy>-\sum_{k=1}^n \frac{\Delta \delta_k}{\overline{\delta}_{k}}\eta_k\geq -\sum_{k=1}^n \eta_k=-\gamma_0,\nonumber
\end{equation}
thus, the smallest eigenvalue of $Q_j$ is greater than $-\gamma_0$, and this completes the proof.
\end{proof}

Proposition 1 divides the smallest eigenvalue of $Q_j$ into $n$ portions and constrains each of them. It directly uses the smallest eigenvalues of the matrices, which can potentially reduce the conservativeness. In addition, the number of constraints in convex problem~\eqref{eq:LMImin2} grows polynomially with respect to the number of buses, $n$. This makes the condition numerically more tractable than methods like Lemma~\ref{lem:2nLMI}. Intuitively, we focus on each of the $n$ CPLs instead of just bounding the largest parameter uncertainty difference. When the bound on every CPL is met, the locally robust stability is ensured. It exempts the need for checking exponentially many constraints, and avoid the potential conservativeness induced by Lemma~\ref{lem:cvx}. Due to its advantages in terms of computational efficiency and conservativeness, it has better applicability for practical DC microgrid applications.  
\subsection{Discussions of the Proposed Framework}\label{sec:robdiscus}
The locally robust stability framework provides a set of sufficient conditions to determine the locally robust stability of system~\eqref{eq:gencls}. In this subsection, we show two insights into the results. 
\subsubsection{Critical Case Interpretation}
To solve problem~\eqref{eq:LMImin} or problem~\eqref{eq:LMImin2}, we only need to know the matrix $\overline{A}_z$ a priori. By definition, the matrix $\overline{A}_z$ can be recognized as the linearized system matrix of the DC microgrid in a case where each CPL's power reaches the upper bound and each CPL's steady state voltage reaches the lower bound. With given constraint sets $\mathcal{P}$ and $\mathcal{V}_l^e$, $\overline{A}_z$ is deterministic and is simple to obtain. We call this case the {\em critical case} of the DC microgrid with constraint sets $\mathcal{P}$ and $\mathcal{V}_l^e$. It is convenient to interpret the locally robust stability framework as checking some properties of the DC microgrid in the critical case (e.g., the Hurwitz stability of the linearized system matrix) to guarantee the stability of the DC microgrid in all possible scenarios. 

Since the critical case depends on the constraint sets, if the sets are modified the corresponding critical cases will differ. Suppose the constraint sets are changed into $\hat{\mathcal{P}}$ and $\hat{\mathcal{V}}_l^e$, respectively, we study the equilibria of system~\eqref{eq:gencls} that lie inside an operationally feasible set $\hat{\mathcal{X}}_e(\hat{\mathcal{P}},\hat{\mathcal{V}}_l^e)$. Let the linearized system matrix of the critical case after the change be given by 
\begin{equation}
\hat{A}_z=A+\sum_{k=1}^n \hat{\delta}_kD_{3n+k},\nonumber
\end{equation}
where $\hat{\delta}_k$ is the new upper bound of each $\delta_k$ with respect to the modified constraint sets $\hat{\mathcal{P}}$ and $\hat{\mathcal{V}}_l^e$. It is worth mentioning that the system parameters (e.g., CPL capacitance $C_{lk}$) and topology will not change with the two sets, i.e. the matrix $A$ does not depends on $\hat{\mathcal{P}}$ and $\hat{\mathcal{V}}_l^e$.

Recall the definition of locally robust stability, we call system~\eqref{eq:gencls} locally robustly stable with $\hat{\mathcal{X}}_e(\hat{\mathcal{P}},\hat{\mathcal{V}}_l^e)$ if any $x_e\in \hat{\mathcal{X}}_e(\hat{\mathcal{P}},\hat{\mathcal{V}}_l^e)$ is locally exponentially stable. Since solving problem~\eqref{eq:LMImin} or problem~\eqref{eq:LMImin2} only requires to know the linearized system matrix of the critical case, to simplify discussion we use $\textbf{P}_1(\overline{A}_z)$ and $\textbf{P}_2(\overline{A}_z)$ to represent problem~\eqref{eq:LMImin} and problem~\eqref{eq:LMImin2}, respectively, when $\overline{A}_z$ is utilized in the problems. 

Notice that the feasibility of $\textbf{P}_l(\overline{A}_z)$ is not a sufficient condition for the feasibility of $\textbf{P}_l(\hat{A}_z)$, $l=1$ or 2, and it does not necessarily guarantee the locally robust stability with $\hat{\mathcal{X}}_e(\hat{\mathcal{P}},\hat{\mathcal{V}}_l^e)$. Therefore, one may need to employ the locally robust stability framework whenever there are changes happened to the constraint sets. Even though the conditions given in the framework can be checked efficiently, for a DC microgrid with constantly changing constraint sets the repetitive computational efforts may still be undesirable.

We next derive a condition to guarantee the locally robust stability with $\hat{\mathcal{X}}_e(\hat{\mathcal{P}},\hat{\mathcal{V}}_l^e)$ of system~\eqref{eq:gencls} without solving $\textbf{P}_l(\hat{A}_z)$, $l=1$ or 2, repeatedly when the constraint sets change.

%

\begin{cor}\label{cor:nocheck}
System~\eqref{eq:gencls} is locally robustly stable with $\hat{\mathcal{X}}_e(\hat{\mathcal{P}},\hat{\mathcal{V}}_l^e)$ if the problem $\textbf{P}_l(\overline{A}_z)$ is feasible, $l=1$ or 2, and $\overline{\delta}_k\geq \hat{\delta}_k$, $\forall k=1,\cdots,n$.
\end{cor}

\begin{proof}
For the DC microgrid after the change, its linearized system matrix $A_z$ lies inside a set $\hat{\mathcal{A}}_z$ given as follows,
\begin{equation}
\hat{\mathcal{A}}_z\!=\!\left\{ \!\!A_z\!\!\in\!\!\mathbb{R}^{4n\times 4n}\!:\! A_z\!\!=\!\!A\!\!+\!\!\sum_{k=1}^n \delta_kD_{3n+k},\delta_k\in [0,\hat{\delta}_k]\!\!\right\}.\nonumber
\end{equation} 
Notice that although $\overline{\delta}_k\geq \hat{\delta}_k$ the set $\hat{\mathcal{A}}_z$ does not necessarily belong to $\mathcal{A}_z$ given in~\eqref{eq:Az}, because of the difference in the lower bound of $\delta_k$.

If $\delta_k\in [\underline{\delta_k},\hat{\delta}_k]$, $A_z\in \mathcal{A}_z$. Since problem $\textbf{P}_l(\overline{A}_z)$ is feasible, $l=1$ or 2, $A_z$ is Hurwitz stable.

If $\delta_k\in [0,\underline{\delta_k})$, $A_z$ can be expressed in the similar form of~\eqref{eq:cvxdelA} as follows,
\begin{equation}
A_z=\overline{A}_z-\sum_{k=1}^n \Delta \delta_kD_{3n+k},\nonumber
\end{equation}
where $\Delta \delta_k=\overline{\delta}_k-\delta_k$. By utilizing the similar arguments for the proof of Lemma~\ref{lem:cvx} and Proposition~\ref{prop:lescsv}, $A_z$ can be shown to be Hurwitz stable as well when $\delta_k\in [0,\underline{\delta_k})$. This completes the proof.
\end{proof}

%
%
%
%

Corollary~\ref{cor:nocheck} presents a condition that reduces the repetitive efforts of checking the robust stability framework. When $\textbf{P}_l(\overline{A}_z)$ is feasible for $l=$1 or 2, the key to make the condition satisfied is to guarantee the upper bound for each $\delta_k$ not increased. Furthermore, it implies that when $\overline{\delta}_k$ is smaller the conditions proposed in the paper is more easily to be satisfied, and the DC microgrid is more likely to be locally robustly stable. Recall that $\overline{\delta}_k=\overline{p}_k/(C_{lk}\cdot \underline{v_{lk}^e})^2$. One method is to increase the load capacitance, $C_{lk}$. This provides one explanation for the observations that larger capacitor increases the stability region of DC microgrid~\cite{7182770}. Another method is to decrease the CPL power upper bound, $\overline{p}_k$. In addition, even if $\overline{p}_k$ is increased, by enabling larger CPL voltage lower bound, $\underline{v_{lk}^e}$, for the DC microgrid, the condition can still be satisfied. This gives rise to another DC microgrid design insight such that a higher level of the steady state load voltage is helpful for the stable operation of DC microgrids. It is in line with the common intuition that a higher level of voltage in the DC microgrid is helpful for stability~\cite{7108029}. Notice that the condition given in Corollary~\ref{cor:nocheck} does not impose any additional constraint on the lower bound of $\delta_k$. Recall that $\underline{\delta_k}=C_{lk}\cdot \underline{p_k}/(\overline{v}_{lk}^e)^2$, the result implies that the values of $\underline{p_k}$ and $\overline{v}_{lk}^e$ have less influences on the locally robust stability of the DC microgrid.

\subsubsection{Application to Distributed Control}\label{sec:robdifcon}
Recently, some distributed controllers have been proposed to replace the V-I droop controllers for DC microgrids~\cite{6816073,6287047}. The modeling method presented in this paper and the corresponding stability results can be easily extended to DC microgrids with these distributed controllers. In the following we use an example to demonstrate this.

Similar to that of~\cite{6816073}, we consider a distributed control law for voltage regulation. We adopt a communication network like~\cite{SimpsonPorco20132603}, which enables each source to communicate with the sources on the neighboring buses. The goal of the voltage regulation is to let DC bus voltage $v_{bk}$ track a reference $v_{bk}^*$. To accomplish this goal, the source voltage $v_k$ is controlled to decrease the sum of the tracking errors in its neighborhood. Let $v_b=[v_{b1},\cdots,v_{bn}]^T$, $v^*_b=[v_{b1}^*,\cdots,v_{bn}^*]^T$, $v=[v_1,\cdots,v_n]^T$, and $W$ be the adjacency matrix of the graph $\mathcal{G}$. With the distributed control law, the dynamics of the sources' output voltage are given as follows,
\begin{align}
&\dot{v}_k=g_k\sum_{j=1}^{\mathcal{N}_k} \left(v_{bj}^*-v_{bj}(t)\right),\nonumber\\ &\dot{v}=\sum_{k=1}^n g_kD_{k} W(v_b^*-v_b),\nonumber
\end{align}
where $g_k>0$ is the control gain, and the dynamics of $v_k$ are determined by a weighted sum of the tracking errors in the neighborhood. Regarding the vector equation, the $k^{\text{th}}$ entry of $W(v_b^*-v_b)$ equals $\sum_{j=1}^{\mathcal{N}_k} (v_{bj}^*-v_{bj}(t))$. With the distributed control law, $v$ becomes a control state, and every equilibrium of $v$ makes sure that the weighted sum of the tracking errors to be zero. Seeing that the goal of voltage regulation is accomplished when the equilibrium is stable, the locally robust stability analysis is of significance to study.

The main results of the paper can be applied for the analysis. The state variable of the closed loop system is augmented by the control state $v$. Let $\zeta=[v^T,x^T]^T$ be the augmented state variable. The dynamics of $v$ are linear in $\zeta$, and the dynamics of $x$ has a linear term in $\zeta$ and a nonlinear term concerning the CPLs as well. Thus, the dynamics of $\zeta$ can be characterized by a system similar to system~\eqref{eq:gencls}. The problem can also be formulated as a stability analysis problem of a polytopic uncertain linear system with uncertainties lying on the diagonal entries of the system matrix. We can apply the proposed robust stability framework for this problem.

\section{Case Study}\label{sec:case}
%
\subsection{Simulation Validation}

In this subsection, we demonstrate the effectiveness of the robust stability framework proposed in Section~\ref{sec:rob}. 

The case study is based on a nine-bus DC microgrid shown in Fig.~\ref{fig:9s9l}. Each source is under the V-I droop control, and the output voltage is given by $v_k=v^{\text{ref}}_k-d_ki_{sk}$, $k=1,\cdots,9$. The parameters of the $k^{\text{th}}$ bus are shown in TABLE~\ref{table:spec_2}. Suppose that each load power is unknown and may take values in the range $[5\text{kW},20\text{kW}]$. Let $p^n_k$ be the nominal power of the $k^{\text{th}}$ CPL, and it is given by $p^n_k=15$kW. In addition, the operationally feasible range of each steady state CPL voltage $v^e_{lk}$ is set as $[360\text{V},440\text{V}]$. The simulation studies are conducted on a laptop with Intel Core i7. The model of the DC microgrid is built and simulated by MATLAB/Simulink, and the convex problems given in the proposed robust stability framework are solved by MATLAB convex optimization toolbox CVX~\cite{cvx,gb08}.  Notice that any other simulation and optimization tools can be used for the case studies.

\begin{figure}[h]
\centering
\includegraphics[width=0.45\textwidth]{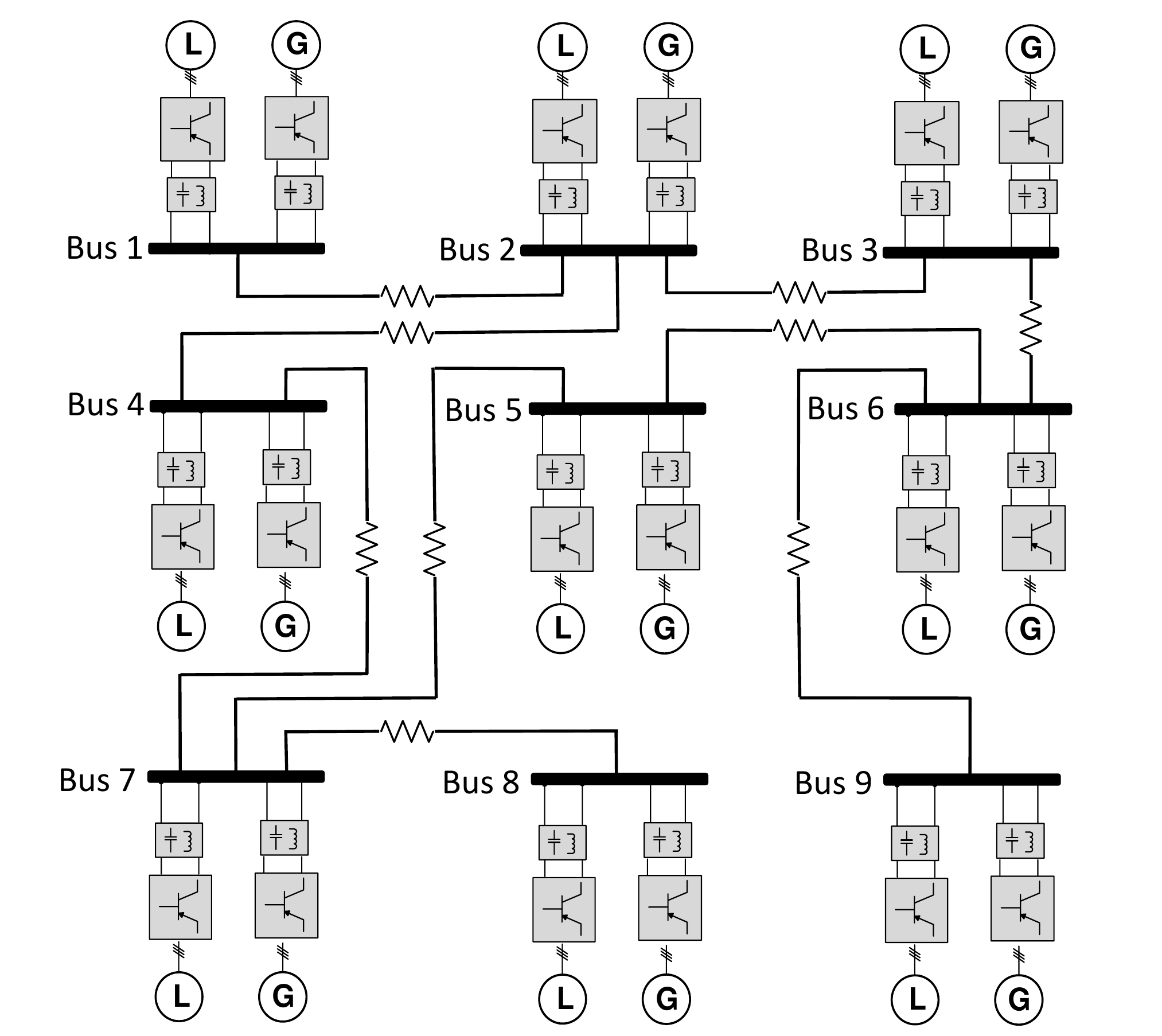}
\caption{\label{fig:9s9l}The nine-bus DC microgrid}
\end{figure}

\begin{table}[h]
\caption{Specifications of the five-bus DC microgrid}
\label{table:spec_2}
\begin{center}
\begin{tabular}{cccc}
\hline\hline
$r_{sk}$ & 0.05 $\Omega$ & $r_{lk}$ & 0.05 $\Omega$ \\ 
$L_{sk}$ & 0.9 mH & $L_{lk}$ & 0.9 mH \\ 
$C_{bk}$ & 0.75 mF & $C_{lk}$ & 0.7 mF \\   
$r_{jk}$ & 1 $\Omega$ & $p^n_k$ & 15 kW\\
$\underline{p_k}$ & 5 kW & $\overline{p}_k$ & 20 kW \\
$\underline{v^e_{lk}}$ & 360 V & $\overline{v^e_{lk}}$ & 440 V \\ 
$v^{\text{ref}}_k$ & 400V\\
\hline\hline
\end{tabular}
\end{center}
\end{table}

With given constrains on the power and steady state voltage of the CPLs, we study the stability of the DC microgrid's equilibria that lie in an operationally feasible set as discussed in Section~\ref{sec:robprob}. Existing works only consider the stability of one given equilibrium which requires the full knowledge of the CPL power profiles~\cite{3248412,6909049,7182770,6031929}. Furthermore, their results cannot guarantee the stability of all the possible equilibria in the feasible set. We demonstrate this by using a simulation example. In this example, we apply existing works to the DC microgrid with nominal CPL power. For the resulting deterministic system, all its equilibria can be found. However, with nine CPLs in the microgrid there are generally 512 equilibria. Sorting out the equilibrium that the DC microgrid operates at is difficult since there are no general results for this determination~\cite{7108029}. By using numerical methods~\cite{6909049,7182770}, we find that the DC microgrid will operate at a stable equilibrium when each CPL power equals $15$kW. However, when load power is different from 15kW other equilibria in the feasible set may still be unstable. Suppose that all the CPL power is the same and let each CPL power increase from 5kW to 20kW. The simulation results of the load power and the DC bus voltage of bus 1 are shown in Fig.~\ref{fig:newcase_uns}. From the figure, even though the DC microgrid is stable when the CPL power is 15kW, the microgrid becomes unstable when the load power increases to around 16kW. 

\begin{figure}[h]
\centering
\includegraphics[width=0.4\textwidth]{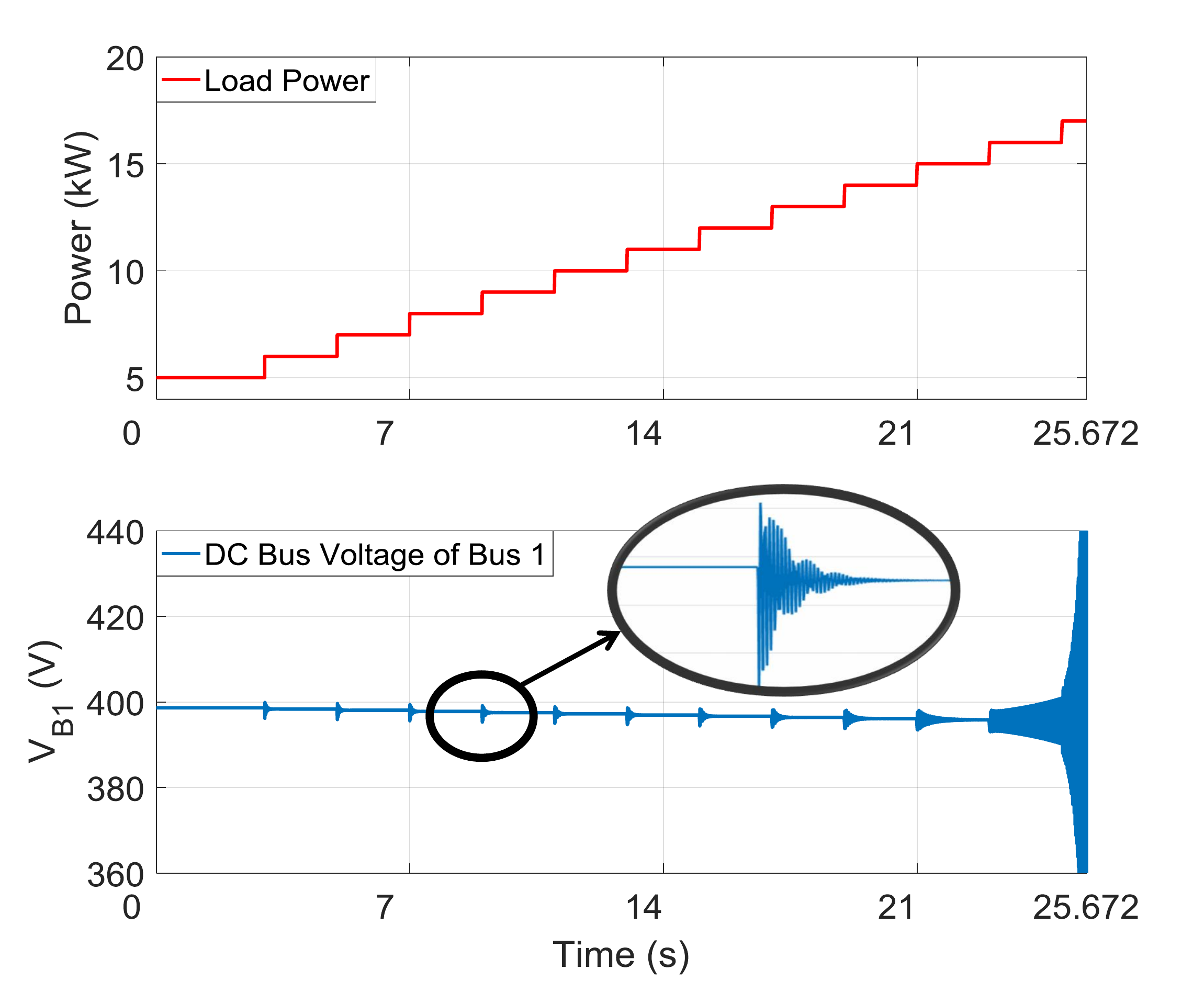}
\caption{\label{fig:newcase_uns}The load power and the DC bus voltage of bus 1 when the droop gains are 0.06}
\end{figure}

\begin{figure}[h]
\centering
\includegraphics[width=0.4\textwidth]{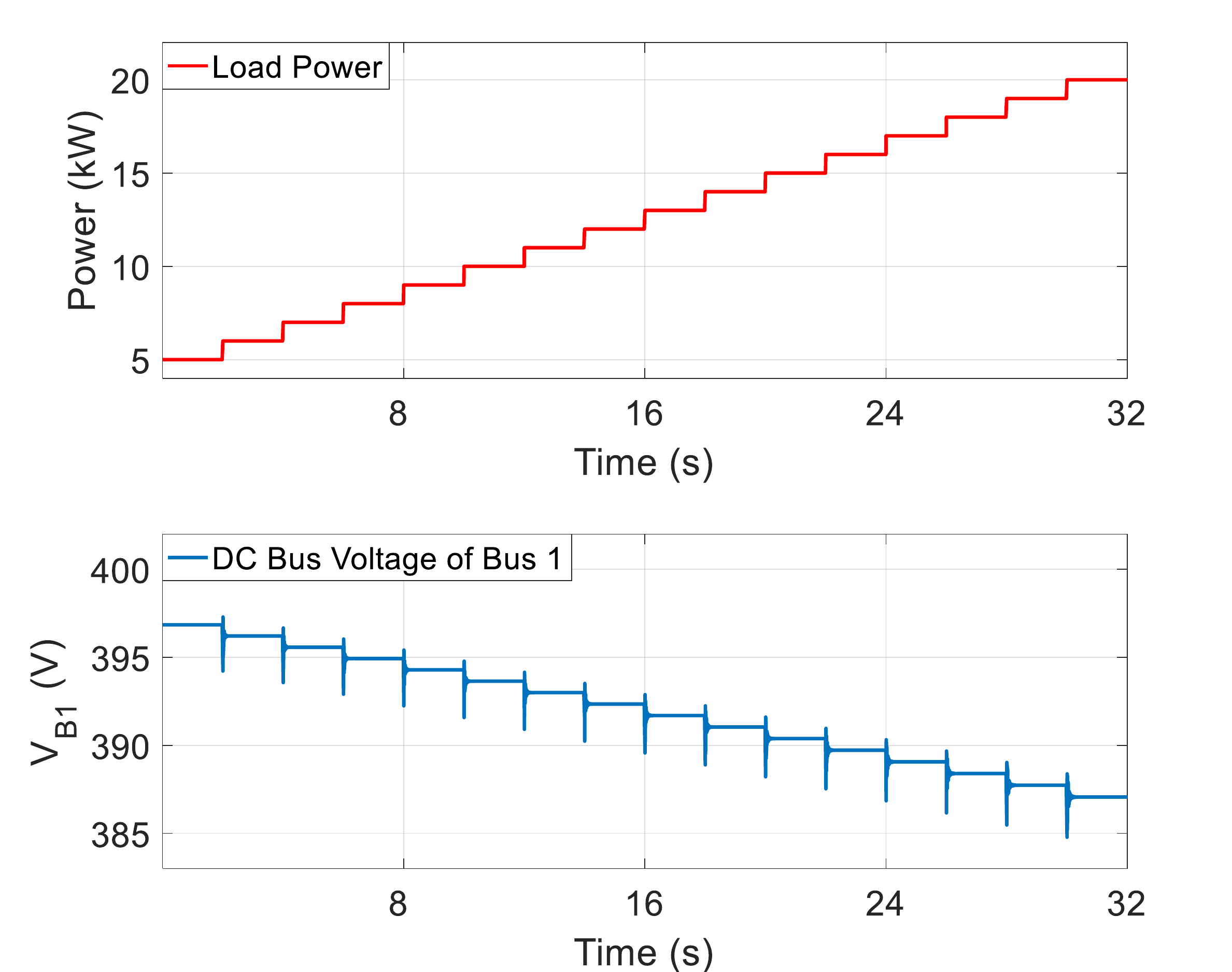}
\caption{\label{fig:newcase_stb}The load power and the DC bus voltage of bus 1 when the droop gains are 0.2}
\end{figure}

In fact, the DC microgrid may operate at every equilibrium in the feasible set. There needs a numerically efficient method to determine the stability of all the operationally feasible equilibria, which gives rise to the application of the proposed robust stability framework. From the results given in Section~\ref{sec:robresult}, the feasibility of either problem~\eqref{eq:LMImin} or problem~\eqref{eq:LMImin2} guarantees the robust stability of the DC microgrid. By using the convex optimization toolbox CVX, neither of problems~\eqref{eq:LMImin} and~\eqref{eq:LMImin2} is feasible when the droop gains are 0.1. This infeasibility happens because not all the equilibria in the feasible set are stable when droop gains are 0.06. However, when the droop gains are set as 0.2, problem~\eqref{eq:LMImin2} is feasible, indicating the DC microgrid will be robustly stable with the increased droop gains. This is verified by the simulation results shown in Fig.~\ref{fig:newcase_stb}. From the figure, the DC microgrid remains stable when the load power varies in the interval [5kW,20kW]. The results demonstrate the effectiveness of the proposed robust stability framework. 

It is worth mentioning that the CVX spends 13.48s to solve the problem.

\subsection{Study on Conservativeness}
In this subsection, we study the conservativeness of the proposed robust analysis framework. We quantify and compare the conservativeness of Lemmas~\ref{lem:2nLMI},~\ref{lem:cvx}, and Proposition~\ref{prop:lescsv}. In addition, we discuss the importance of system parameters to robust stability and the relationship of CPL power with voltage levels. 

Although quantification of the conservativeness of optimization problems are often times difficult, we provide one approximated quantifying method to help clarify the applicability of our work. Recall that the feasibility of problems~\eqref{eq:2n},~\eqref{eq:LMImin}, or~\eqref{eq:LMImin2} in the conditions provides stability certificate to a certain collection of DC microgrids. A condition is said to be more conservative than another when there exists DC microgrid that only the latter can certify. From Section III-C, with the increase of $\overline{\delta}_k$, it becomes more difficult for a condition to hold. The upper bound of $\overline{\delta}_k$ that keeps a condition hold can be used as a quantification of conservativeness. Let $\overline{\delta}_l^u$, $l=1,2,3$, be this upper bound for Lemmas~\ref{lem:2nLMI},~\ref{lem:cvx}, and Proposition~\ref{prop:lescsv}, respectively. It is clear that if $\overline{\delta}_3^u>\overline{\delta}_2^u$, Proposition~\ref{prop:lescsv} not only can certify all the DC microgrid certifiable to Lemma 2 but also is able to certify other DC microgrids. Then we can conclude that Proposition 1 is less conservative than Lemma~\ref{lem:cvx}. 


We use line search methods to find $\overline{\delta}^u_l$ of each condition. One other method is to use $\overline{\delta}_k$ as a decision variable. This makes the problems non-convex due to bilinear couplings. On the other hand, we can use line search methods to gradually approach $\overline{\delta}^u_l$. For example, we can let the initial value of each $\overline{\delta}_k$ be 1, and incrementally increase its value by 1 until a problem becomes unfeasible. 

\begin{figure}[h]
\centering
\includegraphics[width=0.45\textwidth]{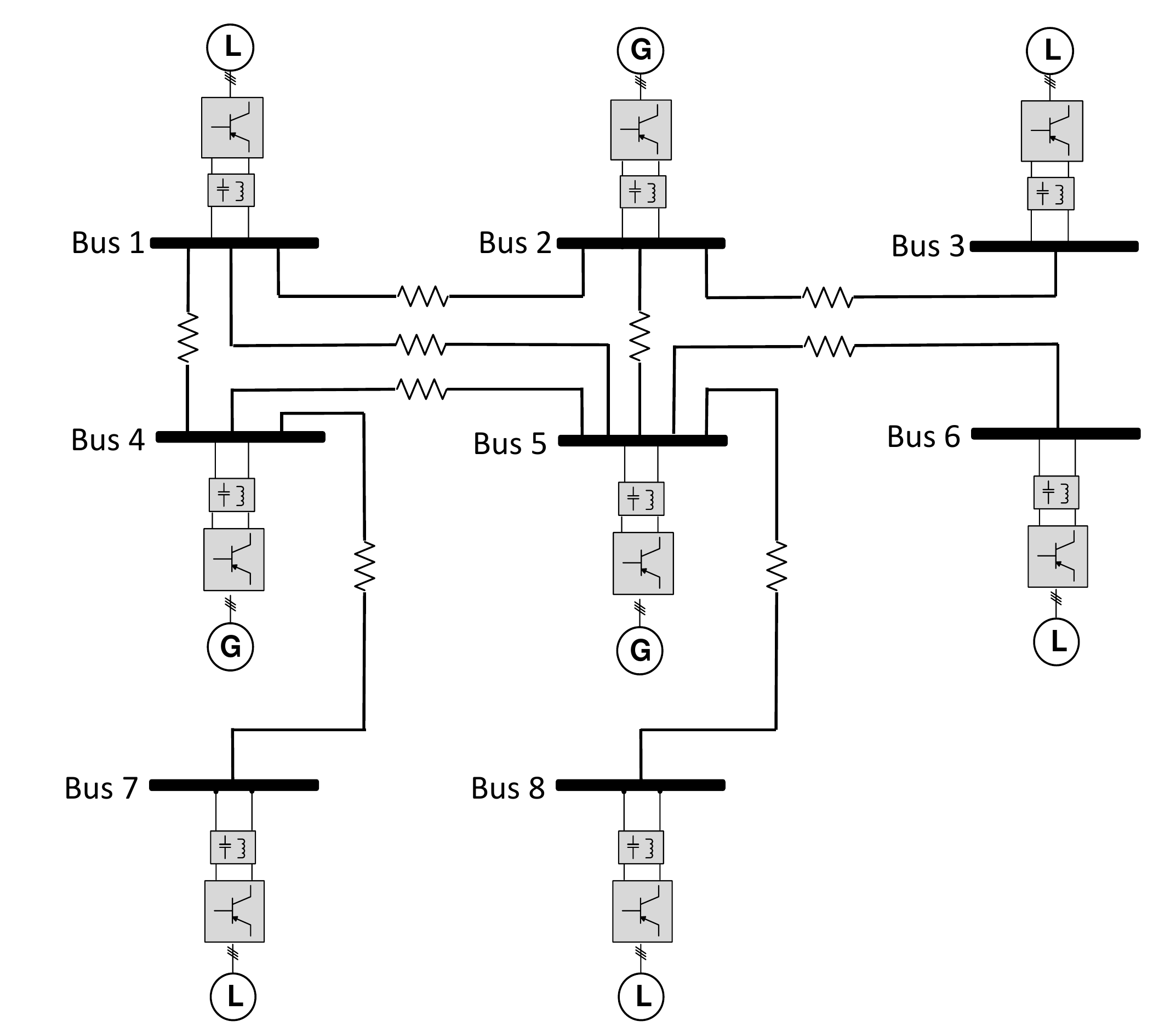}
\caption{\label{fig:8bus}The eight-bus DC microgrid}
\end{figure}

We apply the above described method to an eight-bus DC microgrid obtained from~\cite{Zhao201518}. The topology of the microgrid is shown in Fig.~\ref{fig:8bus}. Notice that there is only one source or load on each bus. Let the droop gains be 0.20 and the system parameters such as inductance, capacitance, and resistance be the same as those in TABLE~\ref{table:spec_2}. The results are given in TABLE~\ref{table:csv_1}. They verify our discussions concerning the conservativeness and computational issues regarding the three conditions in Section~\ref{sec:robresult}. For problem~\eqref{eq:LMImin2}, the results show that it has a good approximation of problem~\eqref{eq:2n}. For problem~\eqref{eq:LMImin}, it is conservative but the least computationally demanding. From the proof in Appendix A, this is because we replace the bound on a decision variable, $\gamma$, to reduce the number of LMIs. For this case study, the replaced bound can be 7 times larger than the previous bound. 

\begin{table}[h]
\caption{Approximated $\overline{\delta}^u_l$ and average computation time for the eight-bus DC microgrid}
\label{table:csv_1}
\begin{center}
\begin{tabular}{cccc}
\hline\hline
& Lemma~\ref{lem:2nLMI} & Lemma~\ref{lem:cvx} & Proposition~\ref{prop:lescsv} \\ 
Approx. $\overline{\delta}^u_l$ & 220 & 51 & 201 \\ 
Ave. Com. Time & 5.3s & 2.6s & 2.8s \\   
\hline\hline
\end{tabular}
\end{center}
\end{table}
\vspace{-0.2cm}
Additionally, one can observe that the quantified conservativeness, $\overline{\delta}^u_l$, depends on the topology and system parameters such as inductance, capacitance, and resistance. By maneuvering the system parameters, the conservativeness of the problems may change subsequently. For example, if we modify the transition resistance to 0.3$\Omega$, then the approximated $\overline{\delta}^u_l$ of problems~\eqref{eq:LMImin2} and~\eqref{eq:2n} become 153 and 170, respectively. This indicates the importance of system parameters to the robust stability of a DC microgrid. 

At last, the quantity $\overline{\delta}^u_l$ also gives practical implications on the relationship of CPL power with voltage levels. Notice that $\overline{\delta}_k$ is a function in the operational bounds of CPL power and voltage, we can find the following inequality,
\begin{equation}
\frac{\overline{p_k}}{C_{lk}(\underline{v^e_{lk}})^2}\leq \overline{\delta}^u_l.\nonumber
\end{equation}
It can be used for practical DC microgrid operations and answers the following two questions: 1. With given maximum CPL power, how should the CPL voltage be regulated to keep stability? 2. With given lowest CPL voltage, how to design the bounds for CPL power for stability?  For example, when the upper bound of CPL power is known, the inequality translates into a minimum certifiable value for the CPL voltage. If the CPL voltage is regulated above this value the microgrid can be ensured to be locally robustly stable. Similarly, we can also find the maximum certifiable CPL power that makes the system stable with given $\underline{v^e_{lk}}$. Regarding the eight-bus DC microgrid with transition resistance being 0.3$\Omega$, $\overline{\delta}^u_1$ of problem~\eqref{eq:2n} is 170. When $\underline{v_{lk}^e}$ is set to be 360V, Lemma~\ref{lem:2nLMI} provides stability certificate for DC microgrids with CPL power less than 15.4kW. To see this, we run simulations with load power increasing from 11kW, the system becomes unstable at around 6s as shown in Fig.~\ref{fig:newcase2_uns}. On the other hand, in order to support 20kW CPL power, the minimum certifiable $\underline{v_{lk}^e}$ is 410V. If we increase the droop references to 440V, from simulation results shown in Fig~\ref{fig:newcase2_stb}, the voltage level of the CPL stays above 410V so that the system can successfully support all the five 20kW CPL. 

\begin{figure}[h]
\centering
\includegraphics[width=0.4\textwidth]{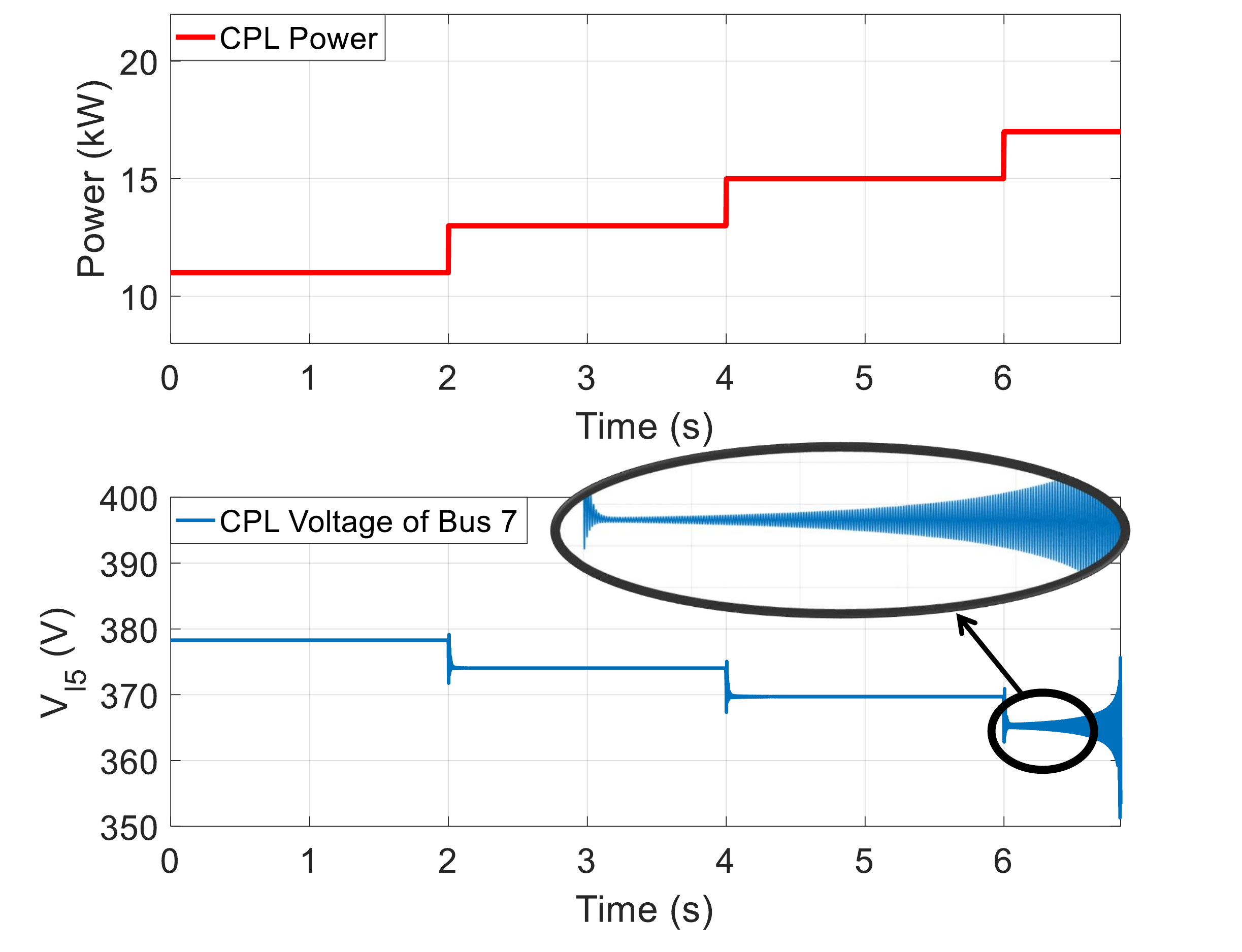}
\caption{\label{fig:newcase2_uns}The load power and the DC bus voltage of bus 7 when the droop references are 400V}
\end{figure}
\begin{figure}[h]
\centering
\includegraphics[width=0.4\textwidth]{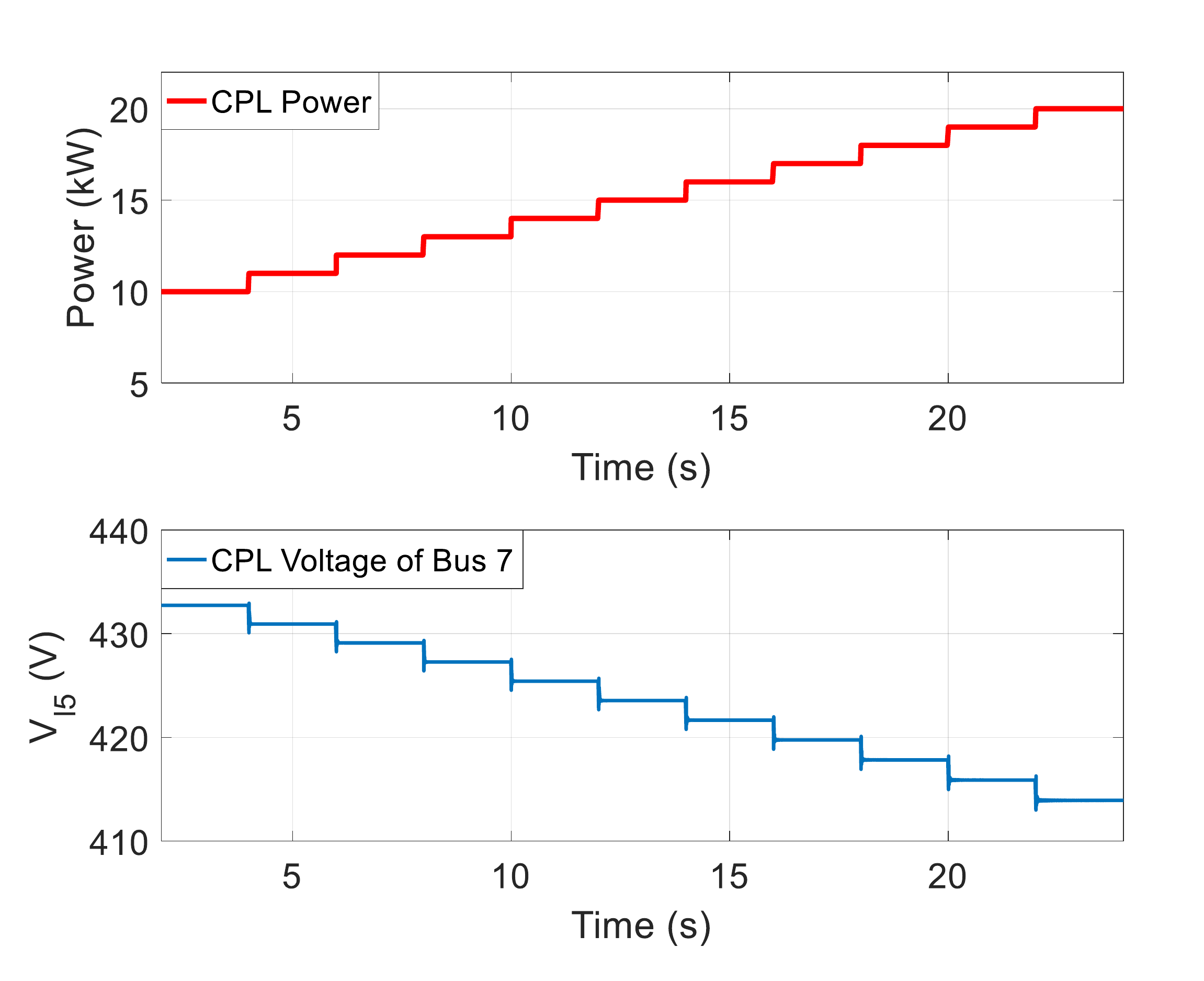}
\caption{\label{fig:newcase2_stb}The load power and the DC bus voltage of bus 7 when the droop references are 440V}
\end{figure}

\section{Conclusion}\label{sec:conclusion}
This paper focuses on stability analysis problem of a general DC microgrid with uncertain CPLs. We study the DC microgrid's equilibria that lie inside an operationally feasible set. When uncertainties are involved, it is difficult to find the actual equilibria that the DC microgrid will converge to. Existing works focus on the stability analysis of a given equilibrium with known CPL power, hence they cannot be applied to this problem. Since the microgrid may operate at any equilibrium in the feasible set, the problem is formulated as a robust stability analysis problem of linear systems with polytopic uncertainties on system matrices. A robust stability framework is proposed in the paper where a set of sufficient conditions are presented. Computationally efficient convex problems are formulated to facilitate the analysis. We use case studies to demonstrate the effectiveness of the results. For future research, we will further study the conservativeness issue in the conditions and the control synthesis for stabilization.
\begin{appendices}
\section{Proof of Lemma 2}
\begin{proof}
We assume that problem~\eqref{eq:LMImin} is feasible. Let $P=P^T\succ 0$, $t>0$, and $\gamma>0$ be the solution of the problem. To prove that system~\eqref{eq:gencls} is locally robustly stable, from Lemma~\ref{lem:2nLMI} it suffices to show that each matrix $PA_j^v+(A^v_j)^TP$ is negative definite, $j=1,\cdots, 2^n$.

Let $\Delta \delta_k=\overline{\delta}_k-\delta_k$ and $\Delta A^v_j=\overline{A}_z-A^v_j$. 

The matrix $A_j^v$ can be expressed in terms of $\overline{A}_z$ and each $\Delta \delta_k$ as follows,
\begin{equation}
A^v_j=\overline{A}_z-\Delta A^v_j=\overline{A}_z-\sum_{k=1}^n \Delta \delta_k D_{3n+k}.\label{eq:cvxdelA}
\end{equation}

Substituting equation~\eqref{eq:cvxdelA} into $PA_j^v+(A^v_j)^TP$ yields,
\begin{equation}
PA_j^v+(A^v_j)^TP=P\overline{A}_z+\overline{A}_z^TP-Q_j.\label{eq:thmtran1}
\end{equation}
Since the matrix $P\overline{A}_z+\overline{A}_z^TP\preceq -\gamma I_{4n}$, if the smallest eigenvalue of $Q_j$ is larger than $-\gamma$, the matrix $PA_j^v+(A^v_j)^TP$ is negative definite. Unfortunately, $Q_j$ is not necessarily positive definite, lower bound on the smallest eigenvalue of $Q_j$ are required and can be provided by upper bounding the largest absolute eigenvalue of $Q_j$ with $\gamma$.

For a matrix $M$, let $||M||$ be the 2 norm of $M$. Recall that if $M$ is symmetric, $||M||$ equals the largest absolute eigenvalue of $M$, and the 2 norm of the diagonal matrix $\sum_{k=1}^n \Delta \delta_k D_{3n+k}$ equals $\overline{\delta}_{\text{max}}$, from equation~\eqref{eq:thmQ} we have
\begin{align}
&||Q_j||\leq 2||P||\cdot ||\sum_{k=1}^n \Delta \delta_k D_{3n+k}||\leq 2t\cdot \overline{\delta}_{\text{max}}<\gamma,\label{eq:thmineq}
\end{align}
hence, the largest absolute eigenvalue of each $Q_j$ is less than $\gamma$, and $PA_j^v+(A_j^v)^TP$ is negative definite, $j=1,\cdots,2^n$. This completes the proof.
\end{proof}
\end{appendices}

\bibliographystyle{ieeetran}
\bibliography{DCMG}
\end{document}